\documentclass[portuges,12pt,letter]{article}
\usepackage[centertags]{amsmath}
\usepackage{amsfonts}
\usepackage{newlfont}
\usepackage{amscd}
\usepackage{graphics}
\usepackage{epsfig}
\usepackage{indentfirst}
\usepackage{amsxtra}
\usepackage[latin1]{inputenc}
\usepackage{amssymb, amsmath}
\usepackage{amsthm}
\usepackage[mathscr]{eucal}

\newtheorem{thm}{Theorem}[section]

\newtheorem{obe}[thm]{Remark}

\date{}
\title{\bf  Dual variational formulations for a large class of non-convex models in the calculus of variations}

\author{Fabio Silva Botelho \\  Department of Mathematics \\
Federal University of Santa Catarina \\
Florian\'{o}polis - SC, Brazil}
\begin{document}
\maketitle

\abstract{This article develops   dual variational formulations for a large class of models in variational optimization. The  results are established through basic tools of functional analysis, convex analysis and duality theory. The main duality principle is developed as an application to a Ginzburg-Landau type system in superconductivity in the absence of a magnetic field. In the first part final sections, we develop  new general   dual convex variational formulations, more specifically, dual formulations with a large region of convexity around the critical points which are suitable for the non-convex optimization for a large class of models in physics and engineering. Finally, in the last section we present some numerical results concerning the generalized method of lines applied to a Ginzburg-Landau type equation.}

\section{Introduction}

In this section we establish a  dual formulation for a large class of models in non-convex optimization.

The main duality principle is applied to the Ginzburg-Landau system in superconductivity in an absence of a magnetic field.

Such results are based on the works of J.J. Telega and W.R. Bielski \cite{2900,85,10,11} and on a D.C. optimization approach developed in Toland \cite{12}.

At this point we start to describe the primal and dual variational formulations.

Let $\Omega \subset \mathbb{R}^3$ be an open, bounded, connected set with a regular (Lipschitzian) boundary denoted by $\partial \Omega.$

For the primal formulation we consider the functional $J:U \rightarrow \mathbb{R}$ where
\begin{eqnarray}
J(u)&=& \frac{\gamma}{2}\int_\Omega \nabla u \cdot \nabla u\;dx \nonumber \\ && + \frac{\alpha}{2} \int_\Omega (u^2-\beta)^2\;dx -\langle u,f \rangle_{L^2}.
\end{eqnarray}

Here we assume $\alpha>0,\beta>0,\gamma>0$, $U=W_0^{1,2}(\Omega)$, $f \in L^2(\Omega)$. Moreover we denote
$$Y=Y^*=L^2(\Omega).$$

Define also $G_1:U \rightarrow \mathbb{R}$ by

$$G_1(u)=\frac{\gamma}{2}\int_\Omega \nabla u  \cdot \nabla u\;dx,$$
$G_2:U\times Y \rightarrow \mathbb{R}$ by

$$G_2(u,v)=\frac{\alpha}{2} \int_\Omega (u^2-\beta+v)^2\;dx+\frac{K}{2}\int_\Omega u^2\;dx,$$
and $F:U \rightarrow \mathbb{R}$ by

$$F(u)=\frac{K}{2}\int_\Omega u^2\;dx,$$
where $K\gg \gamma.$

It is worth highlighting that in such a case
$$J(u)=G_1(u)+G_2(u,0)-F(u)-\langle u,f \rangle_{L^2}, \; \forall u \in U.$$

Furthermore, define the following specific polar functionals specified, namely,
$G_1^*:[Y^*]^2 \rightarrow \mathbb{R}$ by
\begin{eqnarray}G_1^*(v_1^*+z^*)&=&\sup_{ u \in U}\left\{\langle  u,  v_1^*+z^*\rangle_{L^2}-G_1(u)\right\}
\nonumber \\ &=& \frac{1}{2}\int_\Omega [(-\gamma \nabla^2)^{-1} (v_1^*+z^*)](v_1^*+z^*)\;dx,
\end{eqnarray}

$G_2^*:[Y^*]^2 \rightarrow \mathbb{R}$ by
\begin{eqnarray}G_1^*(v_2^*,v_0^*)&=&\sup_{ (u,v) \in U\times Y}\left\{ \langle  u, v_2^*\rangle_{L^2}+\langle  v, v_0^*\rangle_{L^2}-G_2(u,v)\right\}
\nonumber \\ &=& \frac{1}{2}\int_\Omega \frac{(v_2^*)^2}{2v_0^*+K}\;dx \nonumber \\ &&+\frac{1}{2\alpha}\int_\Omega (v_0^*)^2\;dx+\beta \int_\Omega v_0^*\;dx,
\end{eqnarray}
if $v_0^* \in B^*$ where
$$B^*=\{v_0^* \in Y^*\;:\; 2v_0^*+K>K/2 \text{ in } \Omega\},$$
and finally, $F^*:Y^* \rightarrow \mathbb{R}$ by
\begin{eqnarray}F^*(z^*)&=&\sup_{ u \in U}\left\{\langle  u, z^*\rangle_{L^2}-F(u)\right\}
\nonumber \\ &=& \frac{1}{2 K}\int_\Omega (z^*)^2\;dx.
\end{eqnarray}
Define also $$A^*=\{v^*=(v_1^*,v_2^*,v_0^*) \in [Y^*]^2\times B^*\;:\; v_1^*+v_2^*-f=0, \text{ in } \Omega\},$$
$J^*:[Y^*]^4  \rightarrow \mathbb{R}$ by

$$J^*(v^*,z^*)=-G_1^*(v_1^*+z^*)-G_2^*(v_2^*,v_0^*)+F^*(z^*)$$
and $J_1^* :[Y^*]^4\times U \rightarrow \mathbb{R}$ by

$$J_1^*(v^*,z^*,u)=J^*(v^*,z^*)+\langle u,v_1^*+v_2^*-f \rangle_{L^2}.$$

\section{The main duality principle, a convex dual formulation and the concerning proximal primal functional}

Our main result is summarized by the following theorem.

\begin{thm} Considering the definitions and statements in the last section,  suppose also $(\hat{v}^*,\hat{z}^*,u_0) \in [Y^*]^2\times B^* \times Y^* \times U$ is such that $$\delta J_1^*(\hat{v}^*,\hat{z}^*,u_0)=\mathbf{0}.$$
 Under such hypotheses, we have $$\delta J(u_0)=\mathbf{0},$$ $$\hat{v}^* \in A^*$$ and
 \begin{eqnarray}
 J(u_0)&=&\inf_{u \in U}\left\{J(u)+\frac{K}{2}\int_\Omega | u- u_0|^2\;dx\right\} \nonumber \\ &=&
 J^*(\hat{v}^*,\hat{z}^*)  \nonumber \\ &=&
 \sup_{v^* \in A^*}\left\{J^*(v^*,\hat{z}^*)\right\}.
 \end{eqnarray}
\end{thm}
\begin{proof} Since $$\delta J_1^*(\hat{v}^*,\hat{z}^*,u_0)=\mathbf{0}$$  from the variation in $v_1^*$ we obtain
$$ -\frac{(\hat{v}_1^*+\hat{z}^*)}{-\gamma\nabla^2}+ u_0= 0 \text{ in } \Omega,$$ so that
$$\hat{v}_1^*+\hat{z}^*= -\gamma \nabla^2 u_0.$$

From the variation in $v_2^*$ we obtain
$$-\frac{\hat{v}_2^*}{2\hat{v}_0^*+K}+u_0=0, \text{ in } \Omega.$$

From the variation in $v_0^*$ we also obtain
$$\frac{(\hat{v}_2^*)^2}{(2\hat{v}_0^*+K)^2}-\frac{\hat{v}_0^*}{\alpha}-\beta=0$$
and therefore,
$$\hat{v}_0^*=\alpha(u_0^2-\beta).$$

From the variation in $u$ we get
$$\hat{v}_1^*+\hat{v}_2^*-f=0, \text{ in } \Omega$$ and thus
$$\hat{v}^* \in A^*.$$

Finally, from the variation in $z^*$, we obtain
$$-\frac{(\hat{v}_1^*+\hat{z}^*)}{-\gamma \nabla^2}+\frac{\hat{z}^*}{K}=0, \text{ in } \Omega.$$
so that $$-u_0+\frac{\hat{z}^*}{K}=0,$$ that is,
 $$\hat{z}^*=K u_0 \text{ in } \Omega.$$

From such results and $\hat{v}^* \in A^*$ we get
\begin{eqnarray}
0&=&\hat{v}_1^*+\hat{v}_2^*-f \nonumber \\ &=&-\gamma \nabla^2 u_0- \hat{z}^*+2(v_0^*)u_0+Ku_0-f \nonumber \\ &=& -\gamma\nabla^2 u_0+2\alpha(u_0^2-\beta)u_0-f,
\end{eqnarray}
so that
$$\delta J(u_0)=\mathbf{0}.$$

Also from this and from the Legendre transform proprieties we have

$$G_1^*(\hat{v}_1^*+\hat{z}^*)=\left\langle  u_0,  \hat{v}_1^*+\hat{z}^* \right\rangle_{L^2}-G_1(u_0),$$
$$G_2^*(\hat{v}_2^*,\hat{v}_0^*)=\langle  u_0, \hat{v}_2^* \rangle_{L^2}+\langle 0,v_0^* \rangle_{L^2}-G_2(u_0,0),$$
$$F^*(\hat{z}^*)=\langle  u_0,  \hat{z}^* \rangle_{L^2}-F(u_0)$$

and thus we obtain

\begin{eqnarray}
 J^*(\hat{v}^*,\hat{z}^*)
 &=& -G_1^*(\hat{v}_1^*+\hat{z}^*)-G_2^*(\hat{v}_2^*,\hat{v}_0^*)+F^*(\hat{z}^*)
\nonumber \\ &=& -\langle u_0, \hat{v}_1^*+\hat{v}_2^*\rangle +G_1(u_0)+G_2(u_0,0)-F(u_0) \nonumber \\ &=& -\langle u_0,f \rangle_{L^2}+G_1(u_0)+G_2(u_0,0)-F(u_0) \nonumber \\ &=& J(u_0).\end{eqnarray}

Summarizing, we have got
\begin{equation}\label{us5579} J^*(\hat{v}^*,\hat{z}^*)=J(u_0).\end{equation}

On the other hand
\begin{eqnarray}\label{us5580}
 J^*(\hat{v}^*,\hat{z}^*) &=&
-G_1^*(\hat{v}_1^*+\hat{z}^*)-G_2^*(\hat{v}_2^*,\hat{v}_0^*)+F^*(\hat{z}^*) \nonumber \\ &\leq&
- \langle u, \hat{v}_1^*+ \hat{z}^*\rangle_{L^2}-\langle u,\hat{v}_2^*\rangle_{L^2}-\langle 0,v_0^*\rangle_{L^2} +G_1(u)+G_2(u,0)+F^*( \hat{z}^*)
\nonumber \\ &=& -\langle u,f\rangle_{L^2}+G_1(u)+G_2(u,0)-\langle  u, \hat{z}^* \rangle_{L^2}+F^*(\hat{z}^*)
\nonumber \\ &=& -\langle u,f\rangle_{L^2}+G_1(u)+G_2(u,0)-F(u)+F(u)-\langle  u, \hat{z}^* \rangle_{L^2}+F^*(\hat{z}^*)\nonumber \\ &=&J(u)+\frac{K}{2}\int_\Omega u^2\;dx-\langle  u, \hat{z}^* \rangle_{L^2}+F^*(\hat{z}^*) \nonumber \\ &=& J(u)+\frac{K}{2}\int_\Omega u^2\;dx-K\langle  u, u_0 \rangle_{L^2}+\frac{K}{2}\int_\Omega u_0^2\;dx \nonumber \\ &=& J(u)+\frac{K}{2}\int_\Omega | u-u_0|^2\;dx,\; \forall u \in U.\end{eqnarray}

Finally by a simple computation we may obtain the Hessian $$\left\{\frac{\partial^2 J^*(v^*,z^*)}{\partial (v^*)^2}\right\}<\mathbf{0}$$ in $[Y^*]^2 \times B^* \times Y^*$, so that we may infer that $J^*$ is concave in $v^*$ in $[Y^*]^2 \times B^* \times Y^*$.

Therefore, from this, (\ref{us5579}) and
(\ref{us5580}), we have
 \begin{eqnarray}
 J(u_0)&=&\inf_{u \in U}\left\{J(u)+\frac{K}{2}\int_\Omega | u-u_0|^2\;dx\right\} \nonumber \\ &=&
 J^*(\hat{v}^*,\hat{z}^*)  \nonumber \\ &=&
 \sup_{v^* \in A^*}\left\{ J^*(v^*,\hat{z}^*)\right\}.
 \end{eqnarray}
The proof is complete.
\end{proof}
\section{ A primal dual variational formulation}

In this section we develop a more general primal dual variational formulation suitable for a large class of models in non-convex optimization.

Consider again $U=W_0^{1,2}(\Omega)$ and let $G:U \rightarrow \mathbb{R}$ and $F:U \rightarrow \mathbb{R}$ be three times Fr\'{e}chet differentiable  functionals. Let $J:U \rightarrow \mathbb{R}$ be defined by
$$J(u)=G(u)-F(u), \; \forall u \in U.$$

Assume $u_0 \in U$ is such that $$\delta J(u_0)=\mathbf{0}$$ and $$\delta^2 J(u_0) > \mathbf{0}.$$

Denoting $v^*=(v_1^*,v_2^*)$, define $J^*:U \times Y^* \times Y^* \rightarrow \mathbb{R}$ by
\begin{equation}
J^*(u,v^*)= \frac{1}{2}\|v_1^*-G'(u)\|_2^2+\frac{1}{2}\|v_2^*-F'(u)\|_2^2+\frac{1}{2}\|v_1^*-v_2^*\|_2^2\end{equation}

Denoting $L_1^*(u,v^*)=v_1^*-G'(u)$ and $L_2^*(u,v^*)=v_2^*-F'(u)$, define also
$$C^*=\left\{(u,v^*) \in U \times Y^* \times Y^*\;:\; \|L_1^*(u,v_1^*)\|_\infty\leq \frac{1}{K} \text{ and } \|L_2^*(u,v_1^*)\|_\infty\leq \frac{1}{K}\right\},$$
for an appropriate $K>0$ to be specified.

Observe that in $C^*$ the Hessian of $J^*$ is given by
 \begin{equation} \{\delta^2J^*(u,v^*)\}= \left\{
\begin{array}{lrr}
 G''(u)^2+F''(u)^2+\mathcal{O}(1/K) & -G''(u)  & -F''(u)
 \\
 -G''(u) &2 & -1
 \\
 -F''(u) & -1  & 2
 \end{array}\right\} ,\end{equation}

Observe also that $$\det\left\{\frac{\partial^2 J^*(u,v^*)}{\partial v_1^* \partial v_2^*}\right\}=3,$$
and
$$\det\{ \delta^2J^*(u,v^*)\}=(G''(u)-F''(u))^2+\mathcal{O}(1/K)= (\delta^2J(u))^2+\mathcal{O}(1/K).$$

Define now $$\hat{v}_1^*=G'(u_0),$$
$$\hat{v}_2^*=F'(u_0),$$ so that $$\hat{v}_1^*-\hat{v}_2^*=\mathbf{0}.$$

From this we may infer that $(u_0,\hat{v}_1^*,\hat{v}_2^*) \in C^*$ and $$J^*(u_0,\hat{v}^*)=0=\min_{(u,v^*) \in C^*} J^*(u,v^*).$$

Moreover, for $K>0$ sufficiently big, $J^*$ is convex in a neighborhood of $(u_0,\hat{v}^*)$.

Therefore, in the last lines, we have proven the following theorem.

\begin{thm} Under the statements and definitions of the last lines, there exist $r_0>0$ and $r_1>0$ such that
$$J(u_0)= \min_{ u \in B_{r_0}(u_0)} J(u)$$
and
$(u_0,\hat{v}_1^*,\hat{v}_2^*) \in C^*$ is such that $$J^*(u_0,\hat{v}^*)=0=\min_{(u,v^*) \in U \times [Y^*]^2} J^*(u,v^*).$$
Moreover, $J^*$ is convex in $$B_{r_1}(u_0,\hat{v}^*).$$
\end{thm}
\section{ One more  primal dual variational formulation for variational optimization}

Our next result is a new  primal dual variational formulation.

Consider again the functional, as defined in the previous sections, $J:U \rightarrow \mathbb{R}$ where
\begin{eqnarray}J(u)&=& \frac{\gamma}{2}\int_\Omega \nabla u \cdot \nabla u\;dx \nonumber \\ &&+\frac{\alpha}{2} \int_\Omega (u^2-\beta)^2\;dx
-\langle u,f \rangle_{L^2},
\end{eqnarray}
where $\alpha>0,\; \beta>0,\; \gamma>0$, $U=W_0^{1,2}(\Omega)$ $Y=Y^*=L^2(\Omega)$ and $f \in L^\infty(\Omega).$

Define $$B^+=\{u \in U \;:\; uf \geq 0, \text{ in } \Omega\}$$
and
$$E^+=\{u \in U\;:\; \delta^2 J(u)  >\mathbf{0}\}.$$

Observe that

\begin{eqnarray}
J(u)&=&J(u)-\frac{K}{2}\int_\Omega u^4\;dx+ \langle v_0^*,u^2 \rangle_{L^2}
\nonumber \\ &&- \langle v_0^*,u^2 \rangle_{L^2}+\frac{K}{2}\int_\Omega u^4\;dx
\nonumber \\ &\geq&J(u)-\frac{K}{2}\int_\Omega u^4\;dx+ \langle v_0^*,u^2 \rangle_{L^2}
\nonumber \\ && \inf_{u \in U}\left\{- \langle v_0^*,u^2 \rangle_{L^2}+\frac{K}{2}\int_\Omega u^4\;dx\right\} \nonumber \\ &=&
J(u)-\frac{K}{2}\int_\Omega u^4\;dx+ \langle v_0^*,u^2 \rangle_{L^2}-\frac{1}{2K}\int_\Omega (v_0^*)^2\;dx.
\end{eqnarray}

For a critical point we have the relation $$v_0^*=K u^2.$$

With such a relation in mind, we define the following approximate primal dual functional

$$J^*(u,v_0^*)=J(u)+\frac{K}{2}\int_\Omega u^4\;dx-\frac{1}{2K}\int_\Omega (v_0^*)^2\;dx+\frac{K_1}{2} \int_\Omega (v_0^*-Ku^2)^2\;dx,$$
where $K_1\gg K\gg 1.$

Observe also that

\begin{eqnarray}
\det\{\delta^2 J^*(u,v_0^*)\}&=&-\frac{\delta^2 J(u)}{K}+K_1 \delta^2J(u)-6u^2+2K_1v_0^* \nonumber \\ &&
+2K^2K_1^2u^2-2K K_1^2v_0^*.\end{eqnarray}

At this point, we define
$$A_1^+=\{(u,v_0^*) \in U\times Y^*\;:\; Ku^2-v_0^* \leq 0, \text { in } \Omega\}$$

and

$$A_2^+=\{(u,v_0^*) \in U\times Y^*\;:\; Ku^2-v_0^*= 0, \text { in } \Omega\}.$$

Define also $A^+= A_2^+ \cap C^+$, where $$C^+=\{(u,v_0^*) \in U \times Y^*\;:\; u \in  B^+\cap E^+\}.$$

Thus, we have $$\det\{\delta^2 J^*(u,v_0^*)\}=\mathcal{O}(K_1)> \mathbf{0}$$ on  $A+.$

Consider the following convex primal dual problem:

$$\text{ Locally minimize } J^*(u,v_0^*) \text{ on } A^+.$$

Define $\hat{J}^*:U \times Y^* \times Y^* \rightarrow \mathbb{R},$ by

$$\hat{J}^*(u,v_0^*,\lambda)=J^*(u,v_0^*)+\int_\Omega \lambda (Ku^2-v_0^*)\;dx,$$
where $\lambda \in Y^*$ is an appropriate Lagrange multiplier.

Considering such statements and definitions, our main result in this section is summarized by the following theorem.

\begin{thm} Let $(u_0, \hat{v}_0^*) \in A^+$ and $\lambda_0 \in Y^*$  be such that $$\delta \hat{J}^*(u_0,\hat{v}_0^*, \lambda_0)=\mathbf{0}.$$

Under such hypotheses,
$$\delta J(u_0)=\mathbf{0},$$
 and there exists $r>0$ such that

 $$\hat{J}^*(u_0,\hat{v}_0^*,\lambda_0)=J^*(u_0,\hat{v}_0^*)=\min_{(u,v_0^*) \in A^+\cap B_r(u_0,\hat{v}_0^*)} J^*(u,v_0^*)=\min_{u \in B^+\cap E^+}J(u)=J(u_0).$$
 \end{thm}
 \begin{proof}
 Since $(u_0,v_0^*) \in A_2^+$ and $u_0 \in B^+ \cap E^+$, we have that there exists $r>0$ such that
 $$\det\{\delta^2 J^*(u,v_0^*)\}=\mathcal{O}(K_1)> \mathbf{0}$$ in $B_r(u_0,\hat{v}_0^*).$

   From this, $\delta \hat{J}^*(u_0,\hat{v}_0^*,\lambda_0)=\mathbf{0}$, the other hypotheses and the expression for $\lambda_0$ indicated in the next lines,  we have that
 $$J^*(u_0,\hat{v}_0^*)=\min_{(u,v_0^*) \in A^+ \cap B_r(u_0,\hat{v}_0^*)} J^*(u,v_0^*).$$

 From the variation of $\hat{J}^*$ in $u$, we get

 \begin{equation}\label{us2012}\delta J(u_0)+2Ku_0^3+K_1(\hat{v}_0^*-Ku_0^2)(-2Ku_0)+2K\lambda_0 u_0=\mathbf{0}.\end{equation}

 From the variation of $\hat{J}^*$ in $v_0^*$, we have

 $$-\frac{\hat{v}_0^*}{K}+K_1(\hat{v}_0^*-Ku_0^2)-\lambda_0=0,$$
 Finally, from the variation of $\hat{J}^*$ in $\lambda$, we have $$\hat{v}_0^*-Ku_0^2=0$$

 Replacing such results we have obtained
 $$\lambda_0=-u_0$$ so that replacing these last results into (\ref{us2012}), we get
 $$\delta J(u_0)=\mathbf{0}.$$

 Also from $$\hat{v}_0^*-Ku_0^2=0,$$ we have

 \begin{eqnarray}\hat{J}^*(u_0,\hat{v}_0^*,\lambda_0)&=& J^*(u_0,\hat{v}_0^*) \nonumber \\ &=& J(u_0).\end{eqnarray}

 Finally, observe that from similar results in \cite{700}, we may infer that $J$ is convex on the convex set $$B^+ \cap E^+.$$

Joining the pieces, considering that $\delta J(u_0)=\mathbf{0}$, we have got
$$\hat{J}^*(u_0,\hat{v}_0^*,\lambda_0)=J^*(u_0,\hat{v}_0^*)=\min_{(u,v_0^*) \in A^+ \cap B_r(u_0,\hat{v}_0^*)} J^*(u,v_0^*)=\min_{u \in B^+\cap E^+}J(u)=J(u_0).$$

 The proof is complete.

 \end{proof}
\section{ One more duality principle and a concerning  primal dual variational formulation}

In this section we establish a new duality principle and a related  primal dual formulation.

The results are based on the approach of  Toland,  \cite{12}.

\subsection{Introduction}

Let $\Omega \subset \mathbb{R}^3$ be an open, bounded, connected set with a regular (Lipschitzian) boundary denoted by $\partial \Omega.$

Let $J:V \rightarrow \mathbb{R}$ be  a functional such that

$$J(u)=G(u)-F(u), \forall u \in V,$$
where $V=W_0^{1,2}(\Omega)$.

Suppose $G,F$ are both three times Fr\'{e}chet differentiable convex functionals such that $$\frac{\partial^2G(u)}{\partial u^2}>0$$ and
 $$\frac{\partial^2F(u)}{\partial u^2}>0$$ $\forall u \in V.$

 Assume also there exists $\alpha_1 \in \mathbb{R}$ such that

 $$\alpha_1=\inf_{u \in V} J(u).$$

 Moreover, suppose that if $\{u_n\} \subset V$ is such that $$\|u_n\|_V \rightarrow \infty$$ then
 $$J(u_n) \rightarrow +\infty, \text{ as } n \rightarrow \infty.$$

 At this point we define $J^{**}: V \rightarrow \mathbb{R}$ by
 $$J^{**}(u)=\sup_{(v^*, \alpha) \in H^*} \{ \langle u,v^*\rangle +\alpha\},$$
 where $$H^*=\{(v^*,\alpha) \in V^* \times \mathbb{R}\;:\; \langle v,v^*\rangle_V+\alpha \leq F(v),\; \forall v \in V\}.$$

Observe that $(0,\alpha_1) \in H^*,$ so that $$J^{**}(u) \geq \alpha_1=\inf_{u \in V} J(u).$$

On the other hand, clearly we have
$$J^{**}(u) \leq J(u), \; \forall u \in V,$$ so that we have got
$$\alpha_1=\inf_{u \in V} J(u)=\inf_{u \in V} J^{**}(u).$$

Let $u \in V$.

Since $J$ is strongly continuous, there exist $\delta>0$ and $A>0$ such that,

$$\alpha_1 \leq J^{**}(v) \leq J(v) \leq A, \forall v \in B_\delta(u).$$

From this, considering that $J^{**}$ is convex on $V$, we may infer that $J^{**}$ is continuous at $u$, $\forall u \in V.$

Hence $J^{**}$ is strongly lower semi-continuous on $V$, and since $J^{**}$ is convex we may infer that $J^{**}$ is weakly lower semi-continuous on $V$.

Let $\{u_n\} \subset V$ be a sequence such that

$$\alpha_1 \leq  J(u_n) < \alpha_1+\frac{1}{n},\; \forall n \in \mathbb{N}.$$

Hence $$\alpha_1=\lim_{n \rightarrow \infty} J(u_n) = \inf_{u \in V}J(u)=\inf_{u \in V}J^{**}(u).$$

Suppose there exists a subsequence $\{u_{n_k}\}$ of $\{u_n\}$ such that
$$\|u_{n_k}\|_V \rightarrow \infty, \text{ as } k \rightarrow \infty.$$

From the hypothesis we have $$ J(u_{n_k}) \rightarrow + \infty, \text{ as } k \rightarrow \infty,$$ which contradicts $$\alpha_1 \in \mathbb{R}.$$

Therefore there exists $K>0$ such that
$$\|u_n\|_V \leq K,\; \forall u \in V.$$

Since $V$ is reflexive, from this and the Katutani Theorem, there exists a subsequence $\{u_{n_k}\}$ of $\{u_n\}$ and $u_0 \in V$ such that
$$u_{n_k} \rightharpoonup u_0, \text{ weakly in } V.$$

Consequently, from this and considering that $J^{**}$ is weakly lower semi-continuous, we have got
$$\alpha_1=\liminf_{ k \rightarrow \infty}J^{**}(u_{n_k}) \geq J^{**}(u_0),$$ so that
$$J^{**}(u_0)=\min_{u \in V} J^{**}(u).$$

Define $G^*,F^*:V^* \rightarrow \mathbb{R}$ by

$$G^*(v^*)= \sup_{u \in V}\{ \langle u,v^* \rangle_V-G(u)\},$$
and
$$F^*(v^*)= \sup_{u \in V}\{ \langle u,v^* \rangle_V-F(u)\}.$$

Defining also $J^*:V \rightarrow \mathbb{R}$ by $$J^*(v^*)=F^*(v^*)-G^*(v^*),$$ from the results in \cite{12}, we may obtain

$$\inf_{u \in V} J(u)=\inf_{v^* \in V^*}J^*(v^*),$$ so that
\begin{eqnarray}J^{**}(u_0)&=& \inf_{u \in V}J^{**}(u) \nonumber \\ &=& \inf_{u \in V}J(u)=\inf_{v^* \in V^*} J^*(v^*).
\end{eqnarray}

Suppose now there exists $\hat{u} \in V$ such that $$J(\hat{u})=\inf_{u \in V} J(u).$$

From the standard necessary conditions, we have

$$\delta J(\hat{u})=\mathbf{0},$$ so that

$$\frac{\partial G(\hat{u})}{\partial u}-\frac{\partial F(\hat{u})}{\partial u}=\mathbf{0}.$$

Define now $$v_0^*=\frac{\partial F(\hat{u})}{\partial u}.$$

From these last two equations we obtain

$$v_0^*=\frac{\partial G(\hat{u})}{\partial u}.$$

From such results and the Legendre transform properties, we have
$$\hat{u}=\frac{\partial F^*(v_0^*)}{\partial v^*},$$
$$\hat{u}=\frac{\partial G^*(v_0^*)}{\partial v^*},$$
so that $$\delta J^*(v_0^*)=\frac{\partial F^*(v_0^*)}{\partial v^*}-\frac{\partial G^*(v_0^*)}{\partial v^*}=\hat{u}-\hat{u}=\mathbf{0},$$

$$G^*(v^*_0)=\langle \hat{u},v_0^*\rangle_V-G(\hat{u})$$
and
$$F^*(v^*_0)=\langle \hat{u},v_0^*\rangle_V-F(\hat{u})$$ so that
\begin{eqnarray}\inf_{u \in V} J(u)&=&J(\hat{u}) \nonumber \\ &=&G(\hat{u})-F(\hat{u}) \nonumber \\ &=& \inf_{v^* \in V^*}J^*(v^*) \nonumber \\ &=&F^*(v_0^*)-G^*(v_0^*)
\nonumber \\ &=&J^*(v_0^*).\end{eqnarray}

\subsection{The main duality principle and a related  primal dual variational formulation}
Considering these last statements and results, we may prove the following theorem.

\begin{thm}
Let $\Omega \subset \mathbb{R}^3$ be an open, bounded, connected set with a regular (Lipschitzian) boundary denoted by $\partial \Omega.$

Let $J:V \rightarrow \mathbb{R}$ be a functional such that

$$J(u)=G(u)-F(u), \forall u \in V,$$
where $V=W_0^{1,2}(\Omega)$.

Suppose $G,F$ are both three times Fr\'{e}chet differentiable  functionals such that there exists $K>0$ such that$$\frac{\partial^2G(u)}{\partial u^2}+K>0$$ and
 $$\frac{\partial^2F(u)}{\partial u^2}+K>0$$ $\forall u \in V.$

 Assume also there exists $u_0 \in V$ and  $\alpha_1 \in \mathbb{R}$ such that

 $$\alpha_1=\inf_{u \in V} J(u)=J(u_0).$$

 Assume $K_3>0$ is such that $$\|u_0\|_\infty < K_3.$$

 Define $$\tilde{V}=\{u \in V\;:\; \|u\|_\infty \leq K_3\}.$$

 Assume $K_1>0$ is such that if $u \in \tilde{V}$ then $$\max\left\{\|F'(u)\|_\infty, \|G'(u)\|_\infty,\|F''(u)\|_\infty,\; \|F'''(u)\|_\infty, \|G''(u)\|_\infty, \|G'''(u)\|_\infty \right\} \leq K_1.$$

 Suppose also $$K\gg \max\{K_1,K_3\}.$$

 Define $F_K,G_K:V \rightarrow \mathbb{R}$ by

 $$F_K(u)=F(u)+\frac{K}{2} \int_\Omega u^2\;dx,$$
 and

 $$G_K(u)=G(u)+\frac{K}{2} \int_\Omega u^2\;dx,$$
 $\forall u \in V.$

 Define also
 $G_K^*,F_K^*:V^* \rightarrow \mathbb{R}$ by

 $$G_K^*(v^*)=\sup_{ u \in V} \{ \langle u,v^* \rangle_V-G_K(u)\},$$
 and
 $$F_K^*(v^*)=\sup_{ u \in V} \{ \langle u,v^* \rangle_V-F_K(u)\}.$$

 Observe that since $u_0 \in V$ is such that
 $$J(u_0)=\inf_{u \in V} J(u),$$ we have
 $$\delta J(u_0)=\mathbf{0}.$$

 Let $\varepsilon>0$ be a small constant.

 Define $$v_0^*= \frac{\partial F_K(u_0)}{\partial u} \in V^*.$$

 Under such hypotheses, defining $J_1^*: V \times V^* \rightarrow \mathbb{R}$ by
 \begin{eqnarray}
 J_1^*(u,v^*)&=& F_K^*(v^*)-G_K^*(v^*) \nonumber \\ &&+ \frac{1}{2\varepsilon}\left\|\frac{\partial G_K^*(v^*)}{\partial v^*}-u\right\|_2^2
 + \frac{1}{2\varepsilon}\left\|\frac{\partial F_K^*(v^*)}{\partial v^*}-u\right\|_2^2 \nonumber \\ &&+
  \frac{1}{2\varepsilon}\left\|\frac{\partial G_K^*(v^*)}{\partial v^*}-\frac{\partial F_K^*(v^*)}{\partial v^*}\right\|_2^2,
  \end{eqnarray}
  we have
 \begin{eqnarray}
 J(u_0)&=& \inf_{u \in V} J(u)\nonumber \\
 &=& \inf_{ (u,v^*) \in V \times V^*} J_1^*(u,v^*) \nonumber \\ &=& J_1^*(u_0,v_0^*).
 \end{eqnarray}
 \end{thm}
\begin{proof}
Observe that from the hypotheses and the results and statements of the last subsection
$$J(u_0)=\inf_{u \in V} J(u)=\inf_{v^* \in Y^*} J_K^*(v^*)=J_K^*(v_0^*),$$
where $$J_K^*(v^*)=F_K^*(v^*)-G_K^*(v^*), \forall v^* \in V^*.$$

Moreover we have
$$J_1^*(u,v^*)\geq J_K^*(v^*), \forall u \in V,\; v^* \in V^*.$$
Also from hypotheses and the last subsection results,
$$u_0=\frac{\partial F_K^*(v_0^*)}{\partial v^*}=\frac{\partial G_K^*(v_0^*)}{\partial v^*},$$
so that clearly we have
$$J_1^*(u_0,v_0^*)=J_K^*(v_0^*).$$
From these last results, we may infer that
\begin{eqnarray}
J(u_0)&=&\inf_{u \in V} J(u) \nonumber \\ &=& \inf_{v^* \in V^*} J_K^*(v^*)
\nonumber \\ &=& J_K^*(v_0^*)\nonumber \\ &=& \inf_{(u,v^*) \in V \times V^*} J_1^*(u,v^*) \nonumber \\ &=& J_1^*(u_0,v_0^*).
\end{eqnarray}

The proof is complete.

\end{proof}
\begin{obe} At this point we highlight that $J_1^*$ has a large region of convexity around the optimal point $(u_0,v_0^*)$,  for $K>0$ sufficiently large and corresponding $\varepsilon>0$ sufficiently small.

Indeed, observe that for $v^* \in V^*$,
 $$G_K^*(v^*)=\sup_{u \in V} \{\langle u,v^*\rangle_V-G_K(u)\}=\langle \hat{u},v^*\rangle_V-G_K(\hat{u})$$
 where $\hat{u} \in V$ is such that
 $$v^*=\frac{\partial G_K(\hat{u})}{\partial u}=G'(\hat{u})+K \hat{u}.$$

 Taking the variation in $v^*$ in this last equation, we obtain
 $$1=G''(u)  \frac{\partial \hat{u}}{\partial v^*}+K \frac{\partial \hat{u}}{\partial v^*},$$
 so that
 $$\frac{\partial \hat{u}}{\partial v^*}=\frac{1}{G''(u)+K}=\mathcal{O}\left(\frac{1}{K}\right).$$

 From this we get
 \begin{eqnarray}\frac{\partial^2 \hat{u}}{\partial (v^*)^2}&=&-\frac{1}{(G''(u)+K)^2}G'''(u)\frac{\partial \hat{u}}{\partial v^*}
 \nonumber \\ &=&-\frac{1}{(G''(u)+K)^3}G'''(u) \nonumber \\ &=& \mathcal{O}\left(\frac{1}{K^3}\right).\end{eqnarray}

 On the other hand, from the implicit function theorem
 $$\frac{\partial G_K^*(v^*)}{\partial v^*}=u+[v^*-G_K'(\hat{u})]\frac{\partial \hat{u}}{\partial v^*}=u,$$
 so that
 $$\frac{\partial^2 G_K^*(v^*)}{\partial (v^*)^2}=\frac{\partial \hat{u}}{\partial v^*}=\mathcal{O}\left(\frac{1}{K}\right)$$
 and
 $$\frac{\partial^3 G_K^*(v^*)}{\partial (v^*)^3}=\frac{\partial^2 \hat{u}}{\partial (v^*)^2}=\mathcal{O}\left(\frac{1}{K^3}\right).$$

Similarly, we may obtain
$$\frac{\partial^2 F_K^*(v^*)}{\partial (v^*)^2}=\mathcal{O}\left(\frac{1}{K}\right)$$
 and
 $$\frac{\partial^3 F_K^*(v^*)}{\partial (v^*)^3}=\mathcal{O}\left(\frac{1}{K^3}\right).$$

Denoting $$A=\frac{\partial^2 F_K^*(v_0^*)}{\partial (v^*)^2}$$ and $$B=\frac{\partial^2 G_K^*(v_0^*)}{\partial (v^*)^2},$$ we have

$$\frac{\partial^2 J_1^*(u_0,v_0^*)}{\partial (v^*)^2}=A-B+\frac{1}{\varepsilon}\left(2A^2+2B^2-2AB\right),$$
$$\frac{\partial^2 J_1^*(u_0,v_0^*)}{\partial u^2}=\frac{2}{\varepsilon},$$
and
$$\frac{\partial^2 J_1^*(u_0,v_0^*)}{\partial (v^*)\partial u}=-\frac{1}{\varepsilon}(A+B).$$

From this we get

\begin{eqnarray}
\det(\delta^2 J^*(v_0^*,u_0))&=&\frac{\partial^2 J_1^*(u_0,v_0^*)}{\partial (v^*)^2}\frac{\partial^2 J_1^*(u_0,v_0^*)}{\partial u^2}-\left[\frac{\partial^2 J_1^*(u_0,v_0^*)}{\partial (v^*)\partial u}\right]^2 \nonumber \\ &=& 2\frac{A-B}{\varepsilon}+2\frac{(A-B)^2}{\varepsilon^2} \nonumber \\ &=& \mathcal{O}\left(\frac{1}{\varepsilon^2}\right) \nonumber \\ &\gg& \mathbf{0}
\end{eqnarray}
about the optimal point $(u_0,v_0^*).$
\end{obe}

\begin{obe} Denoting again $Y=Y^*=L^2(\Omega)$, we may also define the functionals $F:V \rightarrow \mathbb{R}$ and $G:V \times Y \rightarrow \mathbb{R}$, where
$$F(u)=\frac{\gamma}{2}\int_\Omega \nabla u \cdot \nabla u \;dx+ \frac{K}{2}\int_\Omega u^2\;dx-\langle u,f \rangle_{L^2}$$
and $$G(u,v)=-\frac{\alpha}{2}\int_\Omega (u^2-\beta+v)^2\;dx+\frac{K}{2}\int_\Omega u^2\;dx,$$ and the respective polar functionals
$F:Y^* \rightarrow \mathbb{R}$ and $G^*:[Y^*]^2 \rightarrow \mathbb{R}$ by
\begin{eqnarray}F^*(v^*)&=&\sup_{u \in V}\{\langle u,v^*\rangle_{L^2}-F(u)\} \nonumber \\ &=&\sup_{ u \in V}\left\{\langle u,v^*\rangle_{L^2}-\frac{\gamma}{2}\int_\Omega \nabla u \cdot \nabla u \;dx- \frac{K}{2}\int_\Omega u^2\;dx+\langle u,f \rangle_{L^2}\right\} \nonumber \\ &=& \frac{1}{2}\int_\Omega \frac{(v^*+f)^2}{-\gamma\nabla^2+K}\;dx,\end{eqnarray}
  \begin{eqnarray}G^*(v^*,v_0^*)&=&\sup_{u \in V}\left\{\inf_{v \in Y}\{\langle u,v^*\rangle_{L^2}-G(u,v)\}\right\}\nonumber \\ &=& \sup_{u \in V}\inf_{v \in L^2}\left\{\langle u,v^* \rangle_{L^2}- \langle v,v_0^*\rangle_{L^2} \right. \nonumber \\ && \left.+\frac{\alpha}{2}\int_\Omega (u^2-\beta+v)^2\;dx-\frac{K}{2}\int_\Omega u^2\;dx\right\} \nonumber \\ &=& -\frac{1}{2}\int_\Omega \frac{(v^*)^2}{(2v_0^*-K)}-\frac{1}{2\alpha}\int_\Omega (v_0^*)^2\;dx-\beta \int_\Omega v_0^*\;dx,\end{eqnarray} in $$A^*=\{ v_0^* \in Y^*\;:\; -2v_0^*+K> K/2, \text{ in } \Omega\}.$$
  Finally, we define
  $$J_2^*(v^*,v_0^*)=-F^*(v^*)+G^*(v^*,v_0^*).$$

  Observe that, at first, this functional is not convex.

  However, for an appropriate choice of $K>0$, $K_1>0$ and $K_2>0$, the extended functional $J_3^*:B^* \times A^* \rightarrow \mathbb{R}$  may be convex in $v^*$ in a large region about a critical point and concave in $v_0^*$ on $B^* \times A^*$, where $$B^*=\{v^* \in Y^*\;:\; \|v^*\|_\infty \leq K_2\},$$ and
  \begin{eqnarray} J_3^*(v^*,v_0^*)&=&J_2^*(v^*,v_0^*)+\frac{K_1}{2}\int_\Omega \left|\frac{\partial J_2^*(v^*,v_0^*)}{\partial v^*}\right|^2\;dx \nonumber \\ &=& -\frac{1}{2}\int_\Omega \frac{(v^*+f)^2}{-\gamma \nabla^2+K}\;dx-\frac{1}{2}\int_\Omega \frac{(v^*)^2}{2v_0^*-K}\;dx \nonumber \\ &&-\frac{1}{2\alpha}\int_\Omega (v_0^*)^2\;dx-\beta \int_\Omega v_0^*\;dx \nonumber \\ &&+\frac{K_1}{2}\int_\Omega\left| \frac{v^*+f}{-\gamma \nabla^2+K}
  -\frac{v^*}{-2v_0^*+K}\right|^2\;dx.\end{eqnarray}

 We highlight the appropriate choice of $K$, $K_1$ and $K_2>0$ depends on $\alpha,\beta$ and $\gamma.$

It is also worth emphasizing  the critical points of $J_2^*$ and $J_3^*$ are the same.
\end{obe}

\section{ A convex dual variational formulation for global optimization}

In this section, for $\Omega \subset \mathbb{R}^3$ open, bounded, connected and with a regular boundary $\partial \Omega$, we define again $V=W_0^{1,2}(\Omega)$, $Y=Y^*=L^2(\Omega)$ and $F:V \rightarrow \mathbb{R,}$  $G:V \times Y \rightarrow \mathbb{R}$ by

$$F(u)=\frac{\gamma}{ 2}\int_\Omega \nabla u \cdot \nabla u \;dx+\frac{K}{2}\int_\Omega u^2\;dx-\langle u,f\rangle_{L^2},$$
$$G(u,v)=-\frac{\alpha}{2}\int_\Omega (u^2-\beta+v)^2\;dx+\frac{K}{2}\int_\Omega u^2\;dx,$$
so that \begin{eqnarray}J(u)&=& \frac{\gamma}{2}\int_\Omega \nabla u \cdot \nabla u \;dx \nonumber \\ &&+\frac{\alpha}{2}\int_\Omega(u^2-\beta)^2\;dx-\langle u,f \rangle_{L^2} \nonumber \\ &=&  F(u)-G(u,0),
\end{eqnarray}
where $\gamma>0$, $\alpha>0,\; \beta>0$ and $f \in L^2(\Omega).$

Define also $F^*:Y^* \rightarrow \mathbb{R}$ by
\begin{eqnarray}F^*(v^*)&=&\sup_{ u \in V} \{\langle u,v^*\rangle_{L^2}-F(u)\} \nonumber \\ &=& \frac{1}{2}\int_\Omega \frac{(v^*+f)^2}{-\gamma \nabla^2+K}\;dx, \end{eqnarray}

\begin{eqnarray}G^*(v^*,v_0^*)&=& \sup_{u \in V}\left\{ \inf_{ v \in Y}\left\{\langle u,v^* \rangle_{L^2}-\langle v,v_0^*\rangle_{L^2}-G(u,v)\right\}\right\} \nonumber \\ &=&
-\frac{1}{2}\int_\Omega \frac{(v^*)^2}{2v_0^*-K}\;dx- \frac{1}{2\alpha}\int_\Omega (v_0^*)^2\;dx-\beta \int_\Omega v_0^*\;dx
\end{eqnarray}
if $v_0^* \in B^*,$ where
$$B^*=\{v_0^* \in Y^*\;:\; \|v_0^*\|_\infty < K_1\},$$ for an appropriate constant $K_1>0.$

Moreover define $J^*:[Y^*]^2 \rightarrow \mathbb{R}$ by
$$J^*(v^*,v_0^*)=-F^*(v^*)+G^*(v^*,v_0^*),$$
and
\begin{eqnarray}J_1^*(v^*,v_0^*)&=&J^*(v^*,v_0^*) \nonumber \\ &&+\frac{K_3}{2}\int_\Omega \left(\frac{v^*+f}{-\gamma \nabla^2+K}-\frac{v^*}{-2v_0^*+K} \right)^2\;dx \nonumber \\ && \frac{K_3}{2}\int_\Omega \left( \frac{{v}_0^*}{\alpha}-\left(\frac{ v^*+f}{-\gamma \nabla^2+K}\right)^2 +\beta\right)^2\;dx.
\end{eqnarray}
$$J_3^*(v^*)=\sup_{v_0^* \in B^*} J^*(v^*,v_0^*),$$

Observe that the critical points of $J^*$ and $J_3^*$ are the same.

On the other hand, defining $$D^*=\{v^* \in Y^* \;:\; \|v^*\|_\infty < K_2,\},$$
 $$\tilde{V}=\{ u \in V \;:\; \|u\|_\infty < K_4\},$$
 from the general result in Toland \cite{12}, for appropriate $K>0,$ $K_1>0$ and $K_2>0$, $K_3>0,$ $K_4>0$, we have
$$\inf_{u \in \tilde{V}} J(u)=\inf_{v^* \in D^*}J_3^*(v^*).$$

Since the critical points of $J_1^*$ correspond in an one to one fashion to critical points of $J_3^*$, from the Ekeland variational principle we may obtain a minimizing sequence for $J_3^*$ which corresponds to a minimizing sequence for $J_1^*$ so that, we may obtain
$$\inf_{u \in \tilde{V}} J(u)=\inf_{v^* \in D^*}J_3^*(v^*)=\inf_{(v^*,v_0^*) \in D^* \times B^*}J_1^*(v^*,v_0^*).$$
Considering these last definitions and results, we have obtained a proof of the following theorem.

\begin{thm} Considering the statements in the last lines in this section, let $\hat{v}^* \in D^*$ be such that
$$J_3^*(\hat{v}^*)=\min_{v^* \in D^*} J_3^*(v^*).$$

Assume $u_0 \in V$ and $\hat{v}_0^* \in Y^*$  such that
$$u_0= \frac{\hat{v}^*+f}{\-\gamma \nabla^2+K}$$ and $$v_0^*=\alpha(u_0^2-\beta),$$ are also such that
$$u_0 \in \tilde{V}$$ and $$v_0^* \in B^*.$$
Under such hypotheses, we have
\begin{eqnarray}
J^*_3(\hat{v}^*)&=& \inf_{v^* \in D^*} J_3^*(v^*) \nonumber \\ &=& \inf_{(v^*,v_0^*) \in D^* \times B^*} J_1^*(v^*,v_0^*)
\nonumber \\ &=& J_1^*(\hat{v}^*,\hat{v}_0^*) \nonumber \\ &=& J(u_0) \nonumber \\ &=& \inf_{ u \in \tilde{V}} J(u).
\end{eqnarray}
\end{thm}
\begin{obe} We highlight such a duality principle concerns a convex dual variational formulation suitable for a global optimization of the primal formulation, in the specific sense that, it has a large region of convexity around a critical point. Finally, we highlight such a dual formulation is applicable to  a large class of models in physics and engineering.
\end{obe}
\section{ A final convex dual variational formulation}

In this section, again for $\Omega \subset \mathbb{R}^3$ an open, bounded, connected set with a regular (Lipschitzian) boundary $\partial \Omega$, $\gamma>0,$ $\alpha>0,$ $\beta>0$ and $f \in L^2(\Omega)$,  we denote $F_1: V \times Y \rightarrow \mathbb{R}$, $F_2:V \rightarrow \mathbb{R}$ and $G: V \times Y \rightarrow \mathbb{R}$ by
\begin{eqnarray}F_1(u,v_0^*)&=&\frac{\gamma}{2}\int_\Omega \nabla u \cdot \nabla u\;dx-\frac{K}{2}\int_\Omega u^2\;dx \nonumber \\ &&
+\frac{K_1}{2}\int_\Omega (-\gamma \nabla^2 u+2v_0^*u-f)^2\;dx+\frac{K_2}{2}\int_\Omega u^2\;dx,\end{eqnarray}

$$F_2(u)=\frac{K_2}{2}\int_\Omega u^2\;dx+\langle u,f\rangle_{L^2},$$
and
$$G(u,v)=\frac{\alpha}{2}\int_\Omega (u^2-\beta+v)^2\;dx+\frac{K}{2}\int_\Omega u^2\;dx.$$

We define also $$J_1(u,v_0^*)=F_1(u,v_0^*)-F_2(u)+G(u,0),$$ $$J(u)=\frac{\gamma}{2}\int_\Omega \nabla u \cdot \nabla u\;dx+\frac{\alpha}{2}\int_\Omega (u^2-\beta)^2\;dx-\langle u,f \rangle_{L^2},$$
and
$F_1^*:[Y^*]^3 \rightarrow \mathbb{R},$ $F_2^*: Y^* \rightarrow \mathbb{R},$ and $G^*:[Y^*]^2 \rightarrow \mathbb{R},$ by
\begin{eqnarray}&& F_1^*(v_2^*,v_1^*,v_0^*) \nonumber \\ &=&\sup_{u \in V}\{ \langle u, v_1^*+v_2^* \rangle_{L^2}-F_1(u,v_0^*)\} \nonumber \\ &=&
\frac{1}{2}\int_\Omega \frac{  \left(v_1^* + v_2^* +
   K_1  (-\gamma \nabla^2+ 2 v_0^*)f\right)^2}{(-\gamma\nabla^2 - K + K_2 +K_1(-\gamma\nabla^2+2v_0^*)^2 )}\;dx \nonumber \\ &&-\frac{K_1}{2}\int_\Omega f^2\;dx,
\end{eqnarray}
\begin{eqnarray}F_2^*(v_2^*)&=& \sup_{u \in V}\{ \langle u,v_2^* \rangle_{L^2}-F_2(u)\} \nonumber \\ &=&
\frac{1}{2K_2}\int_\Omega (v_2^*)^2\;dx, \end{eqnarray}
and
\begin{eqnarray}
G^*(v_1^*,v_0^*)&=&\sup_{ (u,v) \in V\times Y}\{ \langle u,v_1^* \rangle_{L^2}-\langle v,v_0^* \rangle_{L^2}-G(u,v)\} \nonumber \\ &=& \frac{1}{2}\int_\Omega \frac{(v_1^*)^2}{2v_0^*+K}\;dx+\frac{1}{2 \alpha}\int_\Omega (v_0^*)^2\;dx
\nonumber \\ &&+\beta \int_\Omega v_0^*\;dx
\end{eqnarray}
if $v_0^* \in B^*$ where
$$B^*=\{v_0^* \in Y^*\;:\; \|v_0^*\|_\infty \leq K/2\}.$$


Finally, we also define $J_1^*:[Y^*]^2 \times B^* \rightarrow \mathbb{R},$
$$J_1^*(v^*_2,v_1^*,v_0^*)=-F_1^*(v_2^*,v_1^*,v_0^*)+F_2^*(v_2^*)-G^*(v_1^*,v_0^*).$$

By computing $\delta^2J_1^*(v_2^*,v_1^*,v_0^*)$ we may obtain that for appropriate $K>0,$ $K_1>0,$ $K_2>0$,   $J_1^*$ in convex in $v_2^*$ on $Y^*$ and it is concave in $(v_1^*,v_0^*)$ around any critical point.

Considering such statements and definitions, we may prove the following theorem.

\begin{thm} Let $(\hat{v}_2^*,\hat{v}_1^*, \hat{v}_0^*) \in Y^* \times Y^* \times B^*$ be such that
$$\delta J_1^*(\hat{v}_2^*, \hat{v}_1^*, \hat{v}_0^*)=\mathbf{0}$$
and $u_0 \in V$ be such that $$u_0=\frac{\hat{v}_1^*+\hat{v}_2^*+K_1(-\gamma \nabla^2+2v_0^*)f}{K_2-K-\gamma \nabla^2+K_1(-\gamma \nabla^2+2\hat{v}_0^*)^2}.$$

Under such hypotheses, we have
$$\delta J(u_0)=\mathbf{0},$$ so that
\begin{eqnarray}J(u_0)&=&\inf_{u \in V}\left\{J(u)+\frac{K_1}{2}\int_\Omega (-\gamma \nabla^2 u +2\hat{v}_0^*u-f)^2\;dx\right\} \nonumber \\ &=&
\inf_{v_2^* \in Y^*} \left\{ \sup_{(v_1^*,v_0^*) \in Y^* \times B^*} J_1^*(v_2^*,v_1^*,v_0^*) \right\} \nonumber \\ &=&
J_1^*(\hat{v}_2^*,\hat{v}_1^*, \hat{v}_0^*).
\end{eqnarray}
\end{thm}
\begin{proof}
Observe that $\delta J_1^*(\hat{v}_2^*, \hat{v}_1^*, \hat{v}_0^*)=\mathbf{0}$ so that, since $J_1^*$ is convex in $v_2^*$ on $Y^*$ and concave in
$(v_1^*,v_0^*)$ on $Y^* \times B^*$, we obtain
$$J_1^*(\hat{v}_2^*, \hat{v}_1^*, \hat{v}_0^*)=\inf_{v_2^* \in Y^*} \left\{ \sup_{(v_1^*,v_0^*) \in Y^* \times B^*} J_1^*(v_2^*,v_1^*,v_0^*) \right\}.$$

Now we are going to show that
$$\delta J(u_0)=\mathbf{0}.$$

From $$\frac{\partial J_1^*(\hat{v}_2^*,\hat{v}_1^*, \hat{v}_0^*)}{\partial v_2^*}=\mathbf{0},$$ we have
$$-u_0+\frac{\hat{v}_2^*}{K_2}=0,$$ and thus $$\hat{v}_2^*=K_2 u_0.$$
From $$\frac{\partial J_1^*(\hat{v}_2^*,\hat{v}_1^*, \hat{v}_0^*)}{\partial v_1^*}=\mathbf{0},$$ we obtain
$$-u_0-\frac{\hat{v}_1^*-f}{2\hat{v}_0^*+K}=0,$$ and thus $$\hat{v}_1^*=-2\hat{v}_0^*u_0-K u_0+f.$$

Finally, denoting $$D=-\gamma \nabla^2 u_0+2 \hat{v}_0^* u_0-f,$$
from  $$\frac{\partial J_1^*(\hat{v}_2^*,\hat{v}_1^*, \hat{v}_0^*)}{\partial v_0^*}=\mathbf{0},$$ we have
$$-2Du_0+u_0^2-\frac{\hat{v}_0^*}{\alpha}-\beta=0,$$ so that
\begin{equation}\label{w101}\hat{v}_0^*=\alpha(u_0^2-\beta-2Du_0).\end{equation}
Observe now that $$\hat{v}_1^*+\hat{v}_2^*+K_1(-\gamma \nabla^2+2\hat{v}_0^*)f=(K_2-K-\gamma \nabla^2+K_1(-\gamma \nabla^2+2\hat{v}_0^*)^2)u_0$$ so that
\begin{eqnarray}\label{w102}&&K_2u_0-2\hat{v}_0 u_0-Ku_0+f \nonumber \\ &=&K_2u_0-Ku_0-\gamma \nabla^2u_0 +K_1(-\gamma \nabla^2 +2\hat{v}_0^*)(-\gamma \nabla^2 u_0+2\hat{v}_0^*u_0-f).\end{eqnarray}

The solution for this last system of equations (\ref{w101}) and (\ref{w102}) is obtained through the relations
$$\hat{v}_0^*=\alpha(u_0^2-\beta)$$ and $$-\gamma \nabla^2 u_0+ 2\hat{v}_0^* u_0-f=D=0,$$
so that $$\delta J(u_0)= -\gamma \nabla^2 u_0 + 2\alpha(u_0^2-\beta)u_0-f=0$$ and
$$\delta\left\{ J(u_0)+\frac{K_1}{2}\int_\Omega (-\gamma \nabla^2 u_0 +2\hat{v}_0^*u_0-f)^2\;dx\right\}=0,$$
and hence, from the concerning convexity in $u$ on $V$,
$$J(u_0)=\min_{u \in V}\left\{ J(u)+\frac{K_1}{2}\int_\Omega (-\gamma \nabla^2 u +2\hat{v}_0^*u-f)^2\;dx\right\}.$$

Moreover, from the Legendre transform properties
$$F_1^*(\hat{v}_2^*, \hat{v}_1^*, \hat{v}_0^*)= \langle u_0, \hat{v}_2^*+\hat{v}_1^* \rangle_{L^2}-F_1(u_0, \hat{v}_0^*),$$
$$F_2^*(\hat{v}_2^*)= \langle u_0, \hat{v}_2^* \rangle_{L^2}-F_2(u_0),$$
$$G^*(\hat{v}_1^*,\hat{v}_0^*)= -\langle u_0, \hat{v}_1^* \rangle_{L^2}-\langle 0, \hat{v}_0^* \rangle_{L^2}-G(u_0,0),$$
so that
\begin{eqnarray}
J_1^*(\hat{v}_2^*, \hat{v}_1^*, \hat{v}_0^*)&=& -F_1^*(\hat{v}_2^*, \hat{v}_1^*, \hat{v}_0^*)+F_2^*(\hat{v}_2^*)-G^*(\hat{v}_1^*, \hat{v}_0^*)
\nonumber \\ &=& F_1(u_0,\hat{v}_0^*)-F_2(u_0)+G(u_0,0) \nonumber \\ &=& J(u_0).
\end{eqnarray}

Joining the pieces, we have got
\begin{eqnarray}J(u_0)&=&\inf_{u \in V}\left\{J(u)+\frac{K_1}{2}\int_\Omega (-\gamma \nabla^2 u +2\hat{v}_0^*u-f)^2\;dx \right\}\nonumber \\ &=&
\inf_{v_2^* \in Y^*} \left\{ \sup_{(v_1^*,v_0^*) \in Y^* \times B^*} J_1^*(v_2^*,v_1^*,v_0^*) \right\} \nonumber \\ &=&
J_1^*(\hat{v}_2^*,\hat{v}_1^*, \hat{v}_0^*).
\end{eqnarray}

The proof is complete.

\end{proof}

\section{A related numerical computation through the generalized method of lines }
We start by recalling that the generalized method of lines was originally introduced in the book entitled "Topics on Functional Analysis, Calculus of Variations and Duality" \cite{901}, published in 2011.

Indeed, the present results are extensions and applications of previous ones which have been published since 2011, in books and articles such as \cite{901,909,120,700}. About the Sobolev spaces involved we would mention \cite{1}. Concerning the applications, related models in physics are addressed in \cite{100,101}.

We also emphasize that, in such a method, the domain of the partial differential equation in question is discretized in lines (or more generally, in curves) and the concerning solution is written on these lines as functions of boundary conditions and the domain boundary shape.

In fact, in its previous format, this method consists of an application of a kind of a partial finite differences procedure combined with the Banach fixed point theorem to obtain the relation between two adjacent lines (or curves).

In the present article, we propose an improvement concerning the way we truncate the series solution obtained through an application of the Banach fixed point theorem to find the relation between two adjacent lines. The results obtained are very good even as a typical parameter $\varepsilon>0$ is very small.

In the next lines and sections we develop in details such a numerical procedure.

\subsection{About a concerning improvement for the generalized method of lines}

Let $\Omega \subset \mathbb{R}^2$ where
$$\Omega=\{(r,\theta) \in \mathbb{R}^2\;:\; 1 \leq r \leq 2,\; 0 \leq \theta \leq 2 \pi\}.$$

Consider the problem of solving the partial differential equation
\begin{equation}\left\{\begin{array}{ll}
-\varepsilon\left(\frac{\partial^2 u}{\partial r^2}+\frac{1}{r}\frac{\partial u}{\partial r}+\frac{1}{r^2} \frac{\partial^2 u}{\partial \theta^2}\right)  +\alpha u^3-\beta u=f,& \text{ in } \Omega,
 \\
u=u_0(\theta), & \text{ on } \partial \Omega_1, \\
u=u_f(\theta), & \text{ on } \partial \Omega_2. \end{array} \right.\end{equation}

Here $$\Omega=\{(r,\theta) \in \mathbb{R}^2\;:\; 1\leq r \leq 2,\; 0\leq \theta \leq 2\pi\},$$
$$\partial \Omega_1=\{(1,\theta) \in \mathbb{R}^2\;:\; 0\leq \theta \leq 2\pi\},$$
$$\partial \Omega_2=\{(2,\theta) \in \mathbb{R}^2\;:\; 0\leq \theta \leq 2\pi\},$$
$\varepsilon >0,\;\alpha>0,\beta>0$, and $f\equiv 1, \text{ on } \Omega.$

In a partial finite differences scheme, such a system stands for
$$-\varepsilon\left( \frac{u_{n+1}-2u_n+u_{n-1}}{d^2}+\frac{1}{t_n}\frac{u_n-u_{n-1}}{d} +\frac{1}{t_n^2}\frac{\partial^2 u_n}{\partial \theta^2}\right)+\alpha u_n^3-\beta u_n=f_n,$$
$\forall n \in \{1, \cdots,N-1\},$ with the boundary conditions $$u_0=0,$$ and $$u_N=0.$$

Here $N$ is the number of lines and $d=1/N.$

In particular, for $n=1$ we have
$$-\varepsilon\left( \frac{u_2-2u_1+u_0}{d^2}+\frac{1}{t_1}\frac{(u_1-u_0)}{d} +\frac{1}{t_1^2}\frac{\partial^2 u_1}{\partial
\theta^2}\right)+\alpha u_1^3-\beta u_1=f_1,$$

so that
$$u_1=\left(u_2+u_1+u_0+\frac{1}{t_1}(u_1-u_0)\;d +\frac{1}{t_1^2}\frac{\partial^2 u_1}{\partial
\theta^2}d^2+(-\alpha u_1^3+\beta u_1-f_1)\frac{d^2}{\varepsilon}\right)/3.0,$$

We solve this last equation through the Banach fixed point theorem, obtaining $u_1$ as a function of $u_2.$

Indeed, we may set
$$u_1^0=u_2$$
and
\begin{eqnarray}u_1^{k+1}&=&\left(u_2+u_1^k+u_0+\frac{1}{t_1}(u_1^k-u_0)\;d +\frac{1}{t_1^2}\frac{\partial^2 u_1^k}{\partial
\theta^2}d^2 \right. \nonumber \\ && \left.+(-\alpha (u_1^k)^3+\beta u_1^k-f_1)\frac{d^2}{\varepsilon}\right)/3.0,\end{eqnarray}
$\forall k \in \mathbb{N}.$

Thus, we may obtain
$$u_1=\lim_{k \rightarrow \infty}u_1^k\equiv H_1(u_2,u_0).$$

Similarly, for $n=2$, we have
\begin{eqnarray}u_2&=&\left(u_3+u_2+H_1(u_2,u_0)+\frac{1}{t_1}(u_2-H_1(u_2,u_0))\;d +\frac{1}{t_1^2}\frac{\partial^2 u_2}{\partial
\theta^2}d^2 \right. \nonumber \\ && \left.+(-\alpha u_2^3+\beta u_2-f_2)\frac{d^2}{\varepsilon}\right)/3.0,\end{eqnarray}

We solve this last equation through the Banach fixed point theorem, obtaining $u_2$ as a function of $u_3$ and $u_0.$

Indeed, we may set
$$u_2^0=u_3$$
and
\begin{eqnarray}u_2^{k+1}&=&\left(u_3+u_2^k+H_1(u_2^k,u_0)+\frac{1}{t_2}(u_2^k-H_1(u_2^k,u_0))\;d +\frac{1}{t_2^2}\frac{\partial^2 u_2^k}{\partial
\theta^2}d^2 \right. \nonumber \\ && \left.+(-\alpha (u_2^k)^3+\beta u_2^k-f_2)\frac{d^2}{\varepsilon}\right)/3.0,\end{eqnarray}
$\forall k \in \mathbb{N}.$

Thus, we may obtain
$$u_2=\lim_{k \rightarrow \infty}u_2^k\equiv H_2(u_3,u_0).$$

Now reasoning inductively, having $$u_{n-1}=H_{n-1}(u_n,u_0),$$ we may get
\begin{eqnarray}u_n&=&\left(u_{n+1}+u_{n}+H_{n-1}(u_n,u_0)+\frac{1}{t_n}(u_n-H_{n-1}(u_n,u_0))\;d +\frac{1}{t_n^2}\frac{\partial^2 u_n}{\partial
\theta^2}d^2 \right. \nonumber \\ && \left.+(-\alpha u_n^3+\beta u_n-f_n)\frac{d^2}{\varepsilon}\right)/3.0,\end{eqnarray}

We solve this last equation through the Banach fixed point theorem, obtaining $u_n$ as a function of $u_{n+1}$ and $u_0.$

Indeed, we may set
$$u_n^0=u_{n+1}$$
and
\begin{eqnarray}u_n^{k+1}&=&\left(u_{n+1}+u_n^k+H_{n-1}(u_n^k,u_0)+\frac{1}{t_n}(u_n^k-H_{n-1}(u_n^k,u_0))\;d +\frac{1}{t_n^2}\frac{\partial^2 u_n^k}{\partial
\theta^2}d^2 \right. \nonumber \\ && \left. +(-\alpha (u_n^k)^3+\beta u_n^k-f_n)\frac{d^2}{\varepsilon}\right)/3.0,\end{eqnarray}
$\forall k \in \mathbb{N}.$

Thus, we may obtain
$$u_n=\lim_{k \rightarrow \infty}u_n^k\equiv H_n(u_{n+1},u_0).$$

We have obtained $u_n=H_n(u_{n+1},u_0)$, $\forall n \in \{1,\cdots,N-1\}.$

In particular, $u_{N}=u_f(\theta),$ so that we may obtain
$$u_{N-1}=H_{N-1}(u_N,u_0)=H_{N-1}(0)\equiv F_{N-1}(u_N,u_0)=F_{N-1}(u_f(\theta),u_0(\theta)).$$

Similarly,
$$u_{N-2}=H_{N-2}(u_{N-1},u_0)=H_{N-2}(H_{N-1}(u_N,u_0))=F_{N-2}(u_N,u_0)=F_{N-1}(u_f(\theta), u_0(\theta)),$$
an so on, up to obtaining
$$u_1=H_1(u_2)\equiv F_1(u_N,u_0)=F_1(u_f(\theta),u_0(\theta)).$$

The problem is then approximately solved.

\subsection{ Software in Mathematica for  solving such an equation}
We recall that the equation to be solved is a Ginzburg-Landau type one, where

\begin{equation}\left\{\begin{array}{ll}
-\varepsilon\left(\frac{\partial^2 u}{\partial r^2}+\frac{1}{r}\frac{\partial u}{\partial r}+\frac{1}{r^2} \frac{\partial^2 u}{\partial \theta^2}\right)  +\alpha u^3-\beta u=f,& \text{ in } \Omega,
 \\
u=0, & \text{ on } \partial \Omega_1, \\
u=u_f(\theta), & \text{ on } \partial \Omega_2. \end{array} \right.\end{equation}

Here $$\Omega=\{(r,\theta) \in \mathbb{R}^2\;:\; 1\leq r \leq 2,\; 0\leq \theta \leq 2\pi\},$$
$$\partial \Omega_1=\{(1,\theta) \in \mathbb{R}^2\;:\; 0\leq \theta \leq 2\pi\},$$
$$\partial \Omega_2=\{(2,\theta) \in \mathbb{R}^2\;:\; 0\leq \theta \leq 2\pi\},$$
$\varepsilon >0,\;\alpha>0,\beta>0$, and $f\equiv 1, \text{ on } \Omega.$
In a partial finite differences scheme, such a system stands for
$$-\varepsilon\left( \frac{u_{n+1}-2u_n+u_{n-1}}{d^2}+\frac{1}{t_n}\frac{u_n-u_{n-1}}{d} +\frac{1}{t_n^2}\frac{\partial^2 u_n}{\partial \theta^2}\right)+\alpha u_n^3-\beta u_n=f_n,$$
$\forall n \in \{1, \cdots,N-1\},$ with the boundary conditions $$u_0=0,$$ and $$u_N=u_f[x].$$

Here $N$ is the number of lines and $d=1/N.$

At this point we present the concerning software for an approximate solution.

Such a software is for $N=10$ (10 lines) and $u_0[x]=0.$.
\\
\\
\\
*************************************
\begin{enumerate}
\item $m_8=10$;\; $(N=10 \;lines)$
\item $d=1/m8$;
\item $e_1=0.1$; ($\varepsilon=0.1)$
\item $A=1.0$;
\item $B=1.0$;
\item $For[i=1,i<m8,i++, f[i]=1.0];$ \;$(f\equiv 1, \text{ on } \Omega)$
\item $a=0.0$;
\item $For[i=1,i<m8,i++,$

$Clear[b,u];$

$t[i]=1+i*d;$

$b[x_-]=u[i+1][x];$

\item $For[k=1,k<30,k++,$\;\; $\text{(we have fixed the number of iterations)}$

$z=\left(u[i+1][x]+b[x]+a+\frac{1}{t[i]}(b[x]-a)*d \right.\\
\left.+\frac{1}{t[i]^2}D[b[x],\{x,2\}]*d^2+(-A*b[x]^3+B*u[x]+f[i])*\frac{d^2}{e_1}\right)/3.0;$

$z=\\
Series[z,\{u[i+1][x],0,3\},\{u[i+1]'[x],0,1\},\{u[i+1]''[x],0,1\},
\\
\{u[i+1]'''[x],0,0\},\{u[i+1]''''[x],0,0\}];$

$z=Normal[z],$

$z=Expand[z];$

$b[x_-]=z];$

\item $a_1[i]=z;$

\item $Clear[b];$

\item $u[i+1][x_-]=b[x]$;

\item $a=a_1[i]$\;];

\item $b[x_-]=u_f[x];$

\item $For[i=1,i<m8,i++,$

$A_1=a_1[m8-i];$

$A_1=Series[A_1,\{u_f[x],0,3\},\{u_f'[x],0,1\},\{u_f''[x],0,1\},\{u_f'''[x],0,0\},\{u_f''''[x],0,0\}];$

$A_1=Normal[A_1];$

$A_1=Expand[A_1];$

$u[m8-i][x_-]=A_1;$

$b[x_-]=A_1];$

$Print[u[m8/2][x]];$

\end{enumerate}

*************************************

The numerical expressions for the solutions of the concerning $N=10$ lines are given by
\begin{eqnarray}
u[1][x]&=&0.47352 +0.00691 u_f[x]-0.00459 u_f[x]^2+0.00265 u_f[x]^3+0.00039 (u_f'')[x]
\nonumber \\ &&-0.00058 u_f[x] (u_f'')[x]+0.00050 u_f[x]^2 (u_f'')[x]-0.000181213 u_f[x]^3 (u_f'')[x]\end{eqnarray}

\begin{eqnarray}
u[2][x]&=&0.76763 +0.01301 u_f[x]-0.00863 u_f[x]^2+0.00497 u_f[x]^3+0.00068 (u_f'')[x]
\nonumber \\ &&-0.00103 u_f[x] (u_f'')[x]+0.00088 u_f[x]^2 (u_f'')[x]-0.00034 u_f[x]^3 (u_f'')[x]
\end{eqnarray}

\begin{eqnarray}u[3][x]&=&0.91329 +0.02034 u_f[x]-0.01342 u_f[x]^2+0.00768 u_f[x]^3+0.00095 (u_f'')[x]
\nonumber \\ &&-0.00144 u_f[x] (u_f'')[x]+0.00122 u_f[x]^2 (u_f'')[x]-0.00051 u_f[x]^3 (u_f'')[x] \end{eqnarray}

\begin{eqnarray}u[4][x]&=&0.97125 +0.03623 u_f[x]-0.02328 u_f[x]^2+0.01289 u_f[x]^3 +0.00147331 (u_f'')[x]\nonumber \\ &&-0.00223 u_f[x] (u_f'')[x]+0.00182 uf[x]^2 (u_f'')[x]-0.00074 u_f[x]^3 (u_f'')[x]\end{eqnarray}

\begin{eqnarray}
u[5][x]&=&1.01736 +0.09242 u_f[x]-0.05110 u_f[x]^2+0.02387 u_f[x]^3+0.00211 (u_f'')[x]
\nonumber \\ &&-0.00378 u_f[x] (u_f'')[x]+0.00292 u_f[x]^2 (u_f'')[x]-0.00132 u_f[x]^3 (u_f'')[x]
\end{eqnarray}

\begin{eqnarray}
u[6][x]&=&1.02549 +0.21039 u_f[x]-0.09374 u_f[x]^2+0.03422 u_f[x]^3+0.00147 (u_f'')[x] \nonumber \\ &&-0.00634 u_f[x] (u_f'')[x]+0.00467 u_f[x]^2 (u_f'')[x]-0.00200 u_f[x]^3 (u_f'')[x]
\end{eqnarray}

\begin{eqnarray}
u[7][x]&=&0.93854 +0.36459 u_f[x]-0.14232 u_f[x]^2+0.04058 u_f[x]^3+0.00259 (u_f'')[x]
\nonumber \\ &&-0.00747373 u_f[x] (u_f'')[x]+0.0047969 u_f[x]^2 (u_f'')[x]-0.00194u_f[x]^3 (u_f'')[x]
\end{eqnarray}

\begin{eqnarray}
u[8][x]&=&0.74649 +0.57201 u_f[x]-0.17293 u_f[x]^2+0.02791 u_f[x]^3+0.00353 (u_f'')[x]
\nonumber \\ &&-0.00658 u_f[x] (u_f'')[x]+0.00407 u_f[x]^2 (u_f'')[x]-0.00172 u_f[x]^3 (u_f'')[x]
\end{eqnarray}

\begin{eqnarray}
u[9][x]&=&0.43257 +0.81004 u_f[x]-0.13080 u_f[x]^2+0.00042 u_f[x]^3+0.00294 (u_f'')[x]\nonumber \\ &&-0.00398 u_f[x] (u_f'')[x]+0.00222 u_f[x]^2 (u_f'')[x]-0.00066 u_f[x]^3 (u_f'')[x]
\end{eqnarray}

\section{Conclusion} In the first part of this article we develop duality principles for non-convex variational optimization. In the final concerning sections we propose dual convex formulations suitable for a large class of models in physics and engineering. In the last  article section,  we present an advance concerning the computation of a solution for a partial differential equation through the generalized method of lines. In particular, in its previous versions, we used to truncate the series in $d^2$ however, we have realized the results are much better  by taking line solutions in series for $u_f[x]$ and its derivatives, as it is indicated in the present software.

This is a little difference concerning the previous procedure, but with a great result improvement as the parameter $\varepsilon>0$ is small.

Indeed, with a sufficiently large $N$ (number of lines),  we may obtain very good results even as $\varepsilon>0$ is very small.

\end{document}